\documentclass{amsart}
\usepackage{amssymb}
\usepackage{multirow}

\title{The 4-Regular Edge-Transitive Graphs of Girth 4}
\author{Tomas Boothby}
\author{Matt DeVos}

\usepackage{tikz}
\usetikzlibrary{patterns}
\usepackage{mathrsfs}
\usepackage{comment}
\usepackage{longtable}
\usepackage{multicol}
\usepackage{tabu}

\newtheorem{theorem}{Theorem}[section]

\newtheorem{lemma}[theorem]{Lemma}

\newtheorem{proposition}[theorem]{Proposition}
\newtheorem{observation}[theorem]{Observation}

\theoremstyle{definition}

\newcommand{\br}[1]{\left\{#1\right\}}

\newcommand{\Aut}{\mathrm{Aut}}

\newcommand{\Hea}{\mathrm{Hea}}

\newcommand{\coHea}{{{\Hea}^\mathbf{c}}}

\newcommand{\cF}{\mathcal{F}}

\newcommand{\ZZ}{\mathbb{Z}}

\begin{document}

\begin{abstract}
This paper presents a characterization of edge-transitive graphs which are four regular and have girth four.  This class consists of four infinite families plus four exceptional graphs.  
\end{abstract}


\maketitle

\section{Introduction}

 
In this paper graphs are permitted to be infinite, but are assumed to be simple except where explicitly stated.  We treat edges as unordered pairs of vertices, and call an ordered pair $(u,v)$ of adjacent vertices an \emph{arc}.  An \emph{automorphism} of a graph $X = (V,E)$  is a bijection $\phi : V \rightarrow V$ with the property that $\{u,v\} \in E$  if and only if $\{ \phi(u), \phi(v) \} \in E$.  The set of all automorphisms of $X$ forms a group under composition called the \emph{automorphism} group and denoted $\Aut(X)$.    We say that the graph $X$ is \emph{vertex-transitive} (\emph{edge-transitive, arc-transitive}) if $\Aut(X)$ acts transitively on $V$ ($E$, the set of arcs).  Similarly, we will say that an embedded graph $X$ is \emph{vertex-transitive} (\emph{edge-transitive, arc-transitive}) if the subgroup of $\Aut(X)$ which preserves the faces acts transitively on $V$ ($E$, the set of arcs).

The Platonic Solids discovered by the ancient Greeks give a complete collection of 3-connected vertex-transitive graphs embedded in the sphere with the added property that all faces have the same size.  More generally, there is a well-established theory of maps which provides numerous classification results for embedded graphs in a variety of surfaces.  

Instead of working with embedded graphs, where each edge appears in two (generally short) facial cycles, an alternate problem is to attempt to classify graphs which are equipped with a symmetry condition, such as vertex-transitivity, together with some conditions which imply the existence of short cycles.  For instance, it is well-known that there are just four arc-transitive cubic graphs with girth 6 for which there are more than two 6-cycles through a given edge (the \emph{girth} of a graph is the length of its shortest cycle).  These are the graphs known as:  Heawood,  M\"obius-Kantor, Pappus, and Desargues.  The following result follows easily from this.

\begin{proposition}
\label{g6cubic}
Every finite arc-transitive cubic graph with girth equal to 6 apart from the M\"obius-Kantor and Desargues Graph appears as an arc-transitive embedded graph with all faces of size 6 in either the torus or Klein Bottle.  
\end{proposition}

Feng and Nedela \cite{FN} have extended the above result by showing that apart from the Coxeter Graph, every arc-transitive cubic graph of girth 7 appears as an arc-transitive embedded graph with all faces of size 7.  These results have been extended even further to girth $\le 9$ by Conder and Nedela \cite{CN}.  Although it is not our focus, let us comment that there is a related body of results which serve to describe such graphs in terms of their symmetries; notably, Kutnar and Maru\v{s}i\v{c} \cite{KM} have given a complete classification of arc-transitive cubic graphs of girth 6 by showing that apart from finitely many exceptions, every such graph falls into one of two families of explicit Cayley graphs.




We define $\cF$ to be the class of all connected edge-transitive graphs which are 4-regular and have girth 4.  Our goal in this article will be to prove a result similar to Proposition \ref{g6cubic} but for the class $\cF$.  Although we will not provide a full classification, our results appear to be as complete as reasonably possible.  For a graph $X$ in $\cF$, edge-transitivity implies the existence of a positive number $k$, called the \emph{frequency} of $X$, so that every edge is contained in exactly $k$ cycles of length four.  We let $\cF_k$ denote the subclass of $\cF$ consisting of the graphs with frequency $k$.  Similar to the phenomena in the above proposition, we will find some exceptional graphs with high frequency, and for frequency 2 we will encounter some families of graphs which have natural embeddings in the torus.  


There are four exceptional graphs which appear in our classification, and we pause now to introduce them.  As usual, we will let $K_{n,m}$ denote a complete bipartite graph with bipartite classes of sizes $m$ and $n$.

\bigskip

{\bf Exceptional Graphs in $\cF$}
\begin{enumerate}
\item The graph $K_{4,4}$ is in $\cF_9$.
\item There is a unique (up to isomorphism) graph which can be obtained from $K_{5,5}$ by deleting the edges of a perfect matching.  We denote this graph by $K_{5,5}-M$ and note that it is in $\cF_6$.  
\item Consider the Fano Plane, $PG(2,2)$, and denote the its points and lines by $\mathcal{P}$ and $\mathcal{L}$.  We define $\coHea$ to be the bipartite graph with bipartition $\{ \mathcal{P}, \mathcal{L} \}$ and $p \in \mathcal{P}$ adjacent to $\ell \in \mathcal{L}$ if $p$ and $\ell$ are non-incident in the Fano Plane.  $\coHea$ may also be viewed as the bipartite complement of the incidence graph of the Fano Plane (better known as the Heawood graph) and is  in $\cF_3$.  
\item The 4-dimensional cube, denoted $Q_4$, is in $\cF_3$.
\end{enumerate}

\bigskip

For a graph $Y$ and a vertex $v \in V(Y)$ we let $\delta(v)$ denote the set of edges incident with $v$.  If $G \le \Aut(Y)$ then the \emph{stabilizer of} $v$ is the subgroup $G_v = \{ \phi \in G \mid \phi(v) = v \}$.  Note that $G_v$ has a natural action on $\delta(v)$ and more generally on the set of all $t$-element subsets of $\delta(v)$ for any $t$.  

Next we will introduce the four infinite families of graphs appearing in our classification.  the latter two families are quite simple to describe explicitly.  The former two families we will introduce here, but postpone a detailed explanation on how to construct them to the next two sections.  

\bigskip

{\bf Families of Graphs in $\mathbf{\cF}$}
\begin{enumerate}
\item Let $Y$ be a 4-regular graph, let $G \le \Aut(Y)$ act vertex-transitively, and suppose that for every vertex $v$, the action of $G_v$ on the 2-element subsets of $\delta(v)$ has an orbit $\mathcal{Q}_v$ of size four.   We define $\square(Y,G)$ to be the graph with vertex set $E(Y)$ and edge set $\cup_{v \in V(Y)} \mathcal{Q}_v$.  In Section \ref{squaresec} we will show that every graph of the form $\square(Y,G)$ with girth $\ge 4$ is in $\cF$, and every graph in $\cF_1$ can be constructed in this manner.


\item A \emph{quadrangulation} is an embedded graph with the property that every face is bounded by a cycle of length four.  In Section \ref{quadsec} we will describe a generic construction of graphs in $\cF_2$ by showing that they arise as certain special quadrangulations of the torus, cylinder, and plane.  
\item Let $Y$ be a 4-regular graph and construct a new bipartite graph $2 \times Y$ with bipartition $\{   \{1,2\} \times V(Y), E(Y) \}$, where $(i,v) \in \{1,2\} \times V(Y)$ is adjacent to $e \in E(Y)$ if $e,v$ are incident in the original graph $Y$.  If the graph $Y$ is arc-transitive, it is straightforward to verify that $2 \times Y$ is in $\cF_3$.  
\item We define $C_m^{ (2) }$ ($C_{\infty}^{(2)}$) to be the graph with vertex set $\{1,2\} \times \ZZ/ m \ZZ$ ($\{1,2\} \times \mathbb{Z}$) and two vertices $(i,a), (j,b)$ adjacent if $a-b \in \{ -1, 1 \}$.  The graphs $C_m^{(2)}$ with $m \ge 5$ and $C_{\infty}^{(2)}$ are in $\cF_5$.
\end{enumerate}

\bigskip

With this we can state a simplified version of our main result.  More detail on the graphs in the first two entries can be found in the next two sections.  

\begin{theorem}\label{thm:main}
The following table gives a description (up to isomorphism) of the graphs contained in each  nonempty $\cF_k$.
\begin{center}
\begin{tabular}{c|l}
$k$	&	$\cF_k$	\\
\hline
1	&	Graphs of the form $\square ( Y, G )$		\\
2	&	Quadrangulations of the torus, cylinder, or plane \\
3 	&	$Q_4$, $\coHea$, $2 \times Y$ where $Y$ is arc-transitive and 4-regular \\
5	&	$C_m^{(2)}$ for $m \ge 5$ and $C_{\infty}^{(2)}$	\\
6	&	$K_{5,5} - M$.	\\
9	&	$K_{4,4}$.	 \\
\end{tabular}
\end{center}
\end{theorem}


The remainder of the paper is organized as follows.  The second and third sections will be devoted to describing generic constructions of the infinite families of graphs appearing in the first two entries of the above table.  The fourth section establishes a handful of stability lemmas showing that graphs in $\cF$ with certain subgraphs fall into our classification and then calls on these to complete the proof.

\section{Frequency 1}
\label{squaresec}

In this section we will give a generic construction of graphs in $\cF_1$ by showing that they arise from 4-regular graphs $Y$ equipped with a subgroup $G \le \Aut(Y)$ which acts vertex transitively and has the added property that for every vertex~$v$, the action of $G_v$ on the 2-element subsets of $\delta(v)$ has an orbit of size 4.  This is essentially a negative result for us, as it indicates that the family $\cF_1$ is not likely to have a simple explicit description as does $\cF_k$ for $k \ge 2$.  In particular, our results demonstrate that classifying the graphs in $\cF_1$ would require classifying all 4-regular graphs of girth $\ge 5$ which have certain kinds of vertex stabilizers.

First we will investigate how these vertex stabilizers of interest may act on their incident edges.  Let $H$ be a group acting on a four element set $\{a,b,c,d\}$, and suppose that the action of $H$ on the 2-element subsets of $\{a,b,c,d\}$ has an orbit~$O$ of size four.  It follows from the assumption that $O$ is an orbit that the two 2-element subsets of $\{a,b,c,d\}$ not contained in $O$ must be disjoint.  So, we may assume without loss that $O = \{ \{a,b\}, \{b,c\}, \{c,d\}, \{a,d\} \}$.  Let $\hat{H}$ be the subgroup of permutations of $\{a,b,c,d\}$ associated with an element in $H$.  Since $\hat{H}$ must fix $O$ as a set, we find that $\hat{H}$ is a subgroup of the dihedral group generated by $(abcd)$ and $(ac)$.  Since $H$ acts transitively on $O$, the group $\hat{H}$ must then be one of:
\begin{itemize}
\item The dihedral group generated by $(abcd)$ and $(ac)$
\item The cyclic group generated by $(abcd)$ 
\item The Klein 4-group generated by $(ac)$ and $(bd)$.  
\end{itemize}
A quick check reveals that all three of these groups have the desired properties.  So, for the graphs $Y$ of interest to us, the subgroups $G \le \Aut(Y)$ will have vertex stabilizers acting in one of these three manners.  Next we prove that our operator $\square$ yields graphs in $\cF$.

\begin{observation}
Every graph of the form $\square(Y,G)$ with girth $\ge 4$ is in $\cF$. 
\end{observation}

\begin{proof}
By our assumptions, for every $v \in V(Y)$ the action of $G_v$ on the 2-element subsets of $\delta_Y(v)$ has an orbit $\mathcal{Q}_v$ of size four.  By construction, the graph $X = \square(Y,G)$ has vertex set $E(Y)$ and edge set $\cup_{v \in V(Y)} \mathcal{Q}_v$.  It follows from the above discussion that for every $v \in V(Y)$ the subgraph of $X$ with vertex set $\delta_Y(v)$ and edge set $\mathcal{Q}_v$ is a 4-cycle.  It follows immediately from this that the graph $X$ will be 4-regular and have girth 4.  To prove that $X$ is vertex transitive, we will show that the action of $G$ on $X$ (inherited from its action on $Y$) is edge-transitive.  It follows from the assumption that $G$ acts vertex transitively on $Y$ that $G$ acts transitively on $\{ \mathcal{Q}_v \mid v \in V(Y) \}$.  Combining that property with the assumption that each $\mathcal{Q}_v$ is an orbit of the action of $G_v$ demonstrates that $X$ is edge-transitive, as desired.
\end{proof}

\begin{theorem}
If $X$ is in $\cF_1$, then there exists a $4$-regular vertex-transitive graph $Y$ and a suitable group $G \le \Aut(Y)$ so that $X \cong \square(Y,G)$.
\end{theorem}

\begin{proof}
Let $S$ be the set of all 4-cycles in $X$ and define a graph $Y$ with vertex set $S$ by the rule that distinct vertices $s,s' \in S$ are adjacent in $Y$ if these 4-cycles share a vertex.  Note that every vertex $v$ of $X$ is contained in exactly two 4-cycles $s,s'$ so there is a natural correspondence between vertices of $X$ and edges of $Y$ (ex. here $v$ corresponds to $\{s,s'\}$).  Furthermore, the group $G = \Aut(X)$ has a natural action on $S$ by which $G$ acts on $Y$.  

It follows immediately from the edge-transitivity of $X$ (and the fact that every edge of $X$ is in a unique 4-cycle) that $G$ acts vertex-transitively on $Y$.  Now consider a vertex $s$ of $Y$ and let $v_1,v_2,v_3,v_4$ be the cyclically ordered vertices of $X$ corresponding to $s$ as a 4-cycle in $X$.  It follows from the edge-transitivity of $X$ that the action of the stabilizer of $s$ on $\{v_1,v_2,v_3,v_4\}$ must be transitive on the pairs $\{ v_i, v_{i+1} \}$ (working modulo 4), but cannot map any such pair to either $\{v_1,v_3\}$ or $\{v_2,v_4\}$.  In short, the set $\{ \{v_1, v_2\}, \{v_2,v_3\}, \{v_3,v_4\}, \{v_1,v_4\} \}$ is an orbit of the action of the stabilizer of $s$ on the pairs of $\{v_1,v_2,v_3,v_4\}$.  Now we find $X \cong \square(Y,G)$ as desired.
\end{proof}

\section{Quadrangulations}
\label{quadsec}

In this section, we will offer a generic construction of graphs in $\cF_2$.  In fact, all graphs of this type arise as quadrangulations of surfaces.  So, we will be interested in maps on surfaces for which every vertex is degree 4 and every face has size 4 (i.e a map of type $\{ 4,4 \}$).  We will begin by offering some constructions which give rise to such quadrangulations and then we will prove that our list gives all such graphs.  

Let $Z^2$ denote the graph with vertex set $\mathbb{Z}^2$ and two vertices $(x,y)$ and $(x',y')$ adjacent if $\{ |x - x'|, |y - y'| \} = \{0,1\}$ (so $Z^2$ is the usual graph associated with the integer lattice $\mathbb{Z}_2$, where two vertices are adjacent when they have distance 1).  In particular, the graph $Z^2$ has a natural embedding in the Euclidean plane $\mathbb{R}^2$.  

Let $\Lambda$ be an additive subgroup of $\mathbb{Z}^2$.  Then we may obtain a graph as a quotient of our embedded graph $Z^2$ by $\Lambda$.  More precisely, we let $Z^2 / \Lambda$ be the graph with vertex set consisting of all $\Lambda$-cosets in $\mathbb{Z}^2$, and two vertices adjacent if these $\Lambda$-cosets contain vertices of $Z^2$ which are adjacent.  This quotient also extends naturally to our embedding.  That is, the graph $Z^2 / \Lambda$ appears embedded in the surface $\mathbb{R}^2 / \Lambda$ as a quadrangulation.  Note that the surface $Z^2 / \Lambda$  is the plane if $\Lambda$ is trivial, a cylinder if $\Lambda$ is generated by a single element, and otherwise it is a torus.  

The group $\mathbb{Z}^2$ retains a natural action on this embedded graph by the rule that every $(x,y) \in \mathbb{Z}^2$ acts by sending each $\Lambda$-coset $R$ to $R + (x,y)$.  This automorphism preserves the embedding and acts transitively on the vertices.  The edges of $Z^2 / \Lambda$ fall into two orbits under this action: the \emph{horizontal edges} consisting of all edges of the form $\{ R, R+(0,1) \}$ for a $\Lambda$-coset $R$, and the \emph{vertical edges} consisting of all edges of the form $\{ R, R+(1,0) \}$.  

We have already seen that the embedded graphs $Z^2 / \Lambda$ are vertex transitive and their edges fall into at most two orbits.  So, in order for such a graph to be edge-transitive, we need only find a single automorphism which maps a vertical edge to a horizontal one.  Next we exhibit one special instance and four families where the subgroup $\Lambda$ admits such a function (in each case we offer one such function).  

\bigskip

{\bf Subgroups $\mathbf{\Lambda \le \mathbb{Z}^2}$ for which $\mathbf{Z^2 / \Lambda}$ is edge-transitive}

\begin{tabular}{|l|l|l|}
\hline
Subgroup		&	Surface	&	Edge-transitivity	\\
\hline
$\Lambda = \{ (0,0) \}$	&	plane	&	$(x,y) \rightarrow (y,x)$	\\
\hline
$\Lambda = \langle (a,-a) \rangle$ for $a \neq 0$	&	cylinder	&	$(x,y) \rightarrow (y,x)$	\\
\hline
$\Lambda = \langle (a,b), (-b,a) \rangle$ for $(a,b) \neq (0,0)$	&	torus		&	$(x,y) \rightarrow (-y,x)$	\\
\hline
$\Lambda = \langle (a,b), (b,a) \rangle$ for $a \neq b$	&	torus	&	$(x,y) \rightarrow (y,x)$	\\
\hline
$\Lambda = \langle (a,a), (b,-b) \rangle$ for $a,b \neq 0$	&	torus	&	$(x,y) \rightarrow (y,x)$	\\
\hline
\end{tabular}

\bigskip

Our main goal in this section is to prove the following theorem, showing that the above list is complete.

\begin{theorem}
Every graph in $\cF_2$ is isomorphic to a graph of the form $Z^2/ \Lambda$ where $\Lambda$ appears in the above table.
\end{theorem}

\begin{proof}
Let $X$ be a graph in $\cF_2$.  First suppose that there are two adjacent edges which are contained in at least two common 4-cycles.  In this case $X$ must have a subgraph isomorphic to $K_{3,2}$.   Consider a vertex $v$ of degree 3 in this subgraph and let $e$ be the edge incident with $v$ not in this $K_{3,2}$.  Now $e$ must be contained in two 4-cycles, but this gives us at least five 4-cycles through $v$, which is contradictory.  It follows from this that we may obtain an embedding of our graph by adding a disc to each 4-cycle of $X$.  

Suppose (for a contradiction) that this surface is a Klein Bottle.  Now construct the medial graph $Y$ of $X$ as follows:
The vertex set of $Y$ is the set of midpoints of edges of $X$, and two vertices of $Y$ are adjacent if these edges are consecutive on a face of $X$.  This gives us an embedding of $Y$ as a quadrangulation of the Klein bottle, and since $X$ was edge transitive, it follows that $Y$ is a vertex transitive embedded graph.  However, Babai~\cite{B} has classified the vertex transitive maps on the Klein bottle, and shown that they all satisfy a ``4-stripes constraint'' which implies the existence of short noncontractible cycles.  In Appendix B of \cite{B} he demonstrates that there are just two families of 4-regular vertex-transitive quadrangulations of the Klein Bottle, denoted $B$1 and $B$2.  If $Y$ was one of these, reversing our operations to recover $X$ would result in a graph with either parallel edges or with girth 4 but frequency $>2$.  

Now let us pass to the universal cover of this embedded graph.  This universal cover must be a simple 4-regular quadrangulation, so it must be the standard embedding of $Z^2$ in $\mathbb{R}^2$.  It follows from the theory of surfaces (and the assumption that $X$ is not embedded in the Klein Bottle) that our embedded graph~$X$ is isomorphic to $Z^2 / \Lambda$ embedded in $\mathbb{R}^2 / \Lambda$ where $\Lambda$ is an additive subgroup of $\mathbb{Z}^2$.  If $\Lambda$ is trivial, then we have the first outcome in our table.  

It follows from the standard action of $\mathbb{Z}^2$ that our embedded graph is vertex transitive, and we call these automorphisms \emph{translations}.  The edges of our graph fall into two orbits under these translations which we will call horizontal and vertical.  Since $X$ is edge-transitive, there must also be an automorphism $\phi$ of $X$ which sends a horizontal edge to a vertical edge.  This automorphism lifts to an automorphism $\tilde{\phi}$ of $Z^2$ which interchanges the horizontal and vertical edges.  By composing $\tilde{\phi}$ with a suitable translation, we may obtain an automorphism $\tilde{\psi}$ of $Z^2$ which fixes $(0,0)$ and still interchanges the vertical and horizontal edges.  Furthermore, $\tilde{\psi}$ projects to an automorphism $\psi$ of $X$, so $\tilde{\psi}$ maps $\Lambda$ to itself.  Since $\tilde{\psi}$ must give an isometry of the plane which fixes the origin and maps $\mathbb{Z}^2$ to itself, it must either be a reflection about the line $x=y$ or $x = -y$ or it must be a rotation by $\pm \frac{\pi}{2}$.  We will consider these two cases separately.  

In the latter case let $(a,b)$ be a closest point in $\Lambda \setminus \{ (0,0) \}$ to $(0,0)$.  Now in addition to $(0,0)$ and $(a,b)$, our subgroup $\Lambda$ must also contain $(-b, a)$ and the point $(a-b, a+b)$.  If there were another point in $\Lambda$ contained in the square with these four vertices, then $\Lambda$ would have two points at a distance smaller than $\sqrt{ a^2 + b^2 }$ contradicting our choice of $(a,b)$.  It follows that 
$\Lambda = \langle (a,b), (-b,a) \rangle$ giving us the third case above.

In the remaining case, we may assume that reflection about the line $x=y$ fixes $\Lambda$.  If $\Lambda$ is a subset of the line with equation $x=-y$ or a subset of the line with equation $x=y$ then $X$ will be isomorphic to one appearing in the second entry of our table.  Otherwise, we may choose a point $(a,b) \in \Lambda$ for which $a \neq \pm b$ and subject to this $(a,b)$ is as close to the origin as possible.  Now the points $(0,0)$, $(a,b)$, $(b,a)$, and $(a+b, a+b)$ are all in $\Lambda$.  If the parallelogram with these vertices does not contain any other point in $\Lambda$ then $(a,b)$ and $(b,a)$ generate $\Lambda$ and we have the fourth case in our table.  Otherwise, this parallelogram must contain a point of $\Lambda \setminus \{ (0,0) \}$ which lies on the line $x=y$ and is at least as close to $(0,0)$ as $(a+b, a+b)$.  Choose such a point $(c,c)$ with $c$ as small as possible.  If $c < \frac{a+b}{2}$ then $\Lambda$ will also contain the point $(a-c, b-c)$ and this contradicts our choice of $(a,b)$.  We conclude that $c = \frac{a+b}{2}$ and now the points $(0,0)$, $( \frac{a+b}{2}, \frac{a+b}{2})$, $(\frac{a-b}{2}, \frac{b-a}{2})$, and $(a,b)$ all lie in $\Lambda$ and we have the final case using $( \frac{a+b}{2}, \frac{a+b}{2})$ and $(\frac{a-b}{2}, \frac{b-a}{2})$.
\end{proof}

\section{Stability and the Proof}

In this section we will prove three stability lemmas and then use these to prove the main theorem.  Each of these stability lemmas shows that graphs in $\cF$ with certain local behaviour must be particular graphs or families of graphs appearing in our classification.

In preparation for our first lemma we remind the reader of some well-known facts concerning edge-transitive graphs.  If $X$ is an edge-transitive graph which is not arc-transitive then the arcs of $X$ fall into two orbits under the action of $\Aut(X)$.  By selecting one of these orbits we obtain a directed graph, $\vec{X}$, with the property that these edge orientations are preserved by $\Aut(X)$.  If there is a vertex of $\vec{X}$ with edges directed into and away from it then $X$ is vertex-transitive.  Otherwise, there is a partition of the vertices into $\{V_0, V_1\}$ so that every vertex in $V_0$ is a source and every vertex in $V_1$ is a sink, and $(V_0, V_1)$ forms a bipartition of $X$.

\begin{lemma}\label{lem:K32}
Let $X$ be a graph in $\cF$ with an induced subgraph isomorphic to $K_{3,2}$.  If no induced 
subgraph of $X$ is isomorphic to $K_{4,2}$, then $X \cong K_{5,5} - M$. 
\end{lemma}

\begin{proof}
Choose $v \in V(X)$ and let $\{a,b,c,d\}$ be the neighbours of $v$.  First suppose that every 3 element subset of $\{a,b,c,d\}$ is contained in the neighbourhood of a vertex other than $v$.  Since we may assume $X$ has no subgraph isomorphic to $K_{4,2}$ this gives us a subgraph as follows:

 
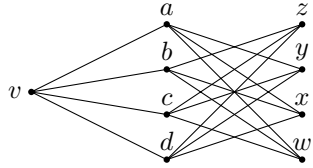
\begin{center}
  \begin{tikzpicture}[scale=.6]
   \coordinate[label=left:{$v$}] (v) at (0,0); \fill (v) circle(2pt);
   \coordinate[label=above:{$a$}] (a) at (3,1.5); \fill (a) circle(2pt);
   \coordinate[label=above:{$b$}] (b) at (3,.5); \fill (b) circle(2pt);
   \coordinate[label=above:{$c$}] (c) at (3,-.5); \fill (c) circle(2pt);
   \coordinate[label=above:{$d$}] (d) at (3,-1.5); \fill (d) circle(2pt);
   \coordinate[label=above:{$w$}] (w) at (6,-1.5); \fill (w) circle(2pt);
   \coordinate[label=above:{$x$}] (x) at (6,-.5); \fill (x) circle(2pt);
   \coordinate[label=above:{$y$}] (y) at (6,.5); \fill (y) circle(2pt);
   \coordinate[label=above:{$z$}] (z) at (6,1.5); \fill (z) circle(2pt);
   \draw (v) to (a) to (w) to (b) to (v) to (c) to (z) to (d) to (v);
   \draw (w) to (c) to (y); \draw (z) to (b) to (x) to (d) to (y) to (a) to (x);
   \end{tikzpicture}
 \end{center} 
By inspection, we find that every 2-edge path containing the edge $\{v,a\}$ is contained in exactly two 4-cycles.  Therefore, by edge-transitivity, every 2-edge-path in $X$ must be contained in exactly two 4-cycles.  Let $u$ be the vertex other than $a,b,c$ which is adjacent to $w$.  Applying this property to the 2-edge-path consisting of the edges $\{u,w\}$ and $\{w,a\}$ we deduce that the vertex $u$ must be adjacent to both $x$ and $y$.  Applying this to the 2-edge-path given by $\{u,w\}$ and $\{w,b\}$ implies that $u$ has neighbourhood $\{w,x,y,z\}$ so $X \cong K_{5,5}-M$.

We may now assume (without loss) that $\{a,b,c\}$ is contained in the neighbourhood of a vertex $w \neq v$ but $\{a,c,d\}$ is not contained in the neighbourhood of any vertex but $v$.  In particular, this implies that $X$ is not arc-transitive (since these assumptions forbid the existence of an automorphism mapping the arc $(v,b)$ to $(v,d)$).  Therefore, we may choose an orientation $\vec{X}$ of our graph which is preserved by $\Aut(X)$.  We may assume (possibly reversing orientations) that $(v,b)$ is an edge of $\vec{X}$, and so $(d,v)$ must also be an edge.  It now follows from edge-transitivity that our directed graph $\vec{X}$ is vertex transitive.  Since every orientation of $K_{3,2}$ contains vertices with indegree at least two and vertices with outdegree at least two we find that $\vec{X}$ is 2-regular.  We may now assume that $(v,b)$ and $(d,v)$ are edges of $\vec{X}$, and then choosing an automorphism mapping $(c,v)$ to $(d,v)$ implies the existence of a vertex $x \in V(X)$ which has $\{a,b,d\}$ as neighbours (in the underlying graph).  This gives us the following.

 
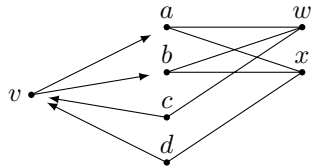
\begin{center}
  \begin{tikzpicture}[scale=.6,>=latex]
   \coordinate[label=left:{$v$}] (v) at (0,0); \fill (v) circle(2pt);
   \coordinate[label=above:{$a$}] (a) at (3,1.5); \fill (a) circle(2pt);
   \coordinate[label=above:{$b$}] (b) at (3,.5); \fill (b) circle(2pt);
   \coordinate[label=above:{$c$}] (c) at (3,-.5); \fill (c) circle(2pt);
   \coordinate[label=above:{$d$}] (d) at (3,-1.5); \fill (d) circle(2pt);
   \coordinate[label=above:{$w$}] (w) at (6,1.5); \fill (w) circle(2pt);
   \coordinate[label=above:{$x$}] (x) at (6,.5); \fill (x) circle(2pt);
   \foreach \a/\b in {v/a,v/b,c/v,d/v} {
    \draw[->,shorten >= 1.5ex] (\a) to (\b);
   }   
   \draw (c) to (w) to (a) to (x) to (b) to (w);
   \draw (d) to (x);
   \end{tikzpicture}
 \end{center}

Our present assumptions imply that there does not exist a vertex other than~$v$ whose neighbourhood contains $\{a,c,d\}$ or $\{b,c,d\}$. It follows from this (and vertex transitivity) that whenever two distinct vertices $p,q$ in $X$ have three common neighbours, these three common neighbours include both out neighbours of $p$ and $q$ in $\vec{X}$.  Applying this property to the vertices $a$ and $b$ (which have $v,w,x$ as common neighbours) implies that we have the following orientations in $\vec{X}$.

 
%
%
%
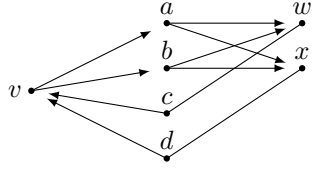
 \begin{center}
  \begin{tikzpicture}[scale=.6,>=latex]
   \coordinate[label=left:{$v$}] (v) at (0,0); \fill (v) circle(2pt);
   \coordinate[label=above:{$a$}] (a) at (3,1.5); \fill (a) circle(2pt);
   \coordinate[label=above:{$b$}] (b) at (3,.5); \fill (b) circle(2pt);
   \coordinate[label=above:{$c$}] (c) at (3,-.5); \fill (c) circle(2pt);
   \coordinate[label=above:{$d$}] (d) at (3,-1.5); \fill (d) circle(2pt);
   \coordinate[label=above:{$w$}] (w) at (6,1.5); \fill (w) circle(2pt);
   \coordinate[label=above:{$x$}] (x) at (6,.5); \fill (x) circle(2pt);
   \foreach \a/\b in {v/a,v/b,a/w,a/x,b/w,b/x,c/v,d/v} {
    \draw[->,shorten >= 1.5ex] (\a) to (\b);
   }
   \foreach \a/\b in {w/c,x/d} {
    \draw (\a) to (\b);
   }
  
   \end{tikzpicture}
 \end{center}
 However, now $w$ and $v$ have $\{a,b,c\}$ as common neighbours, but both $a,b$ are inneighbours of $w$ in $\vec{X}$.  This contradiction completes the proof.
\end{proof}

\begin{lemma}\label{lem:Q4coHea}
Let $X$ be a graph in $\cF_k$ with no subgraph isomorphic to $K_{3,2}$.  If $k \ge 3$ then $X$ is isomorphic to either $Q_4$ or $coHea$.
\end{lemma}

\begin{proof}
We will label vertices with subsets of $\br{1,2,3,4}$ and will write, for example, $12$ to represent $\br{1,2}$.  Label a vertex $\emptyset$ and for each $i \geq 0$, let $L_i = \br{v : \mathrm{dist}(v,\emptyset) = i}$ be the set of vertices at distance $i$ from $\emptyset$.  Label the vertices in $L_1$ with (distinct) singletons 1, 2, 3, and 4.   

Observe that every vertex in $L_2$ is adjacent to at most two vertices in $L_1$ (since $X$ has no subgraph isomorphic to $K_{3,2}$), and similarly, no two vertices in $L_2$ have two common neighbours in $L_1$.  Now for a vertex in $L_2$ with two neighbours in $L_1$, say $a$ and $b$, we assign this vertex the label $ab$.  

By the above, we know that every 2-edge path with vertex sequence $a, \emptyset, b$ for $1 \le a,b \le 4$ completes to at most one 4-cycle.  However, our frequency assumption implies that the edge between $\emptyset$ and $a$ must be contained in $k \ge 3$ such 4-cycles.  Therefore, $k=3$ and every 2-edge path of the form $a, \emptyset, b$ completes to a unique 4-cycle using the vertex labelled $ab$.  So, $L_2$ consists of 6 vertices, each of which is labelled by a distinct 2 element subset.  

Next we observe that every 2-edge path containing the edge between $\emptyset$ and $1$ is contained in a unique 4-cycle.  So, by edge-transitivity, we have the \emph{completion property}: every 2-edge path in $X$ is contained in a unique 4-cycle.  Using this completion property, it is straightforward to prove that $L_2$ must be an independent set (an adjacency of the form $ab \sim bc$ would create a 3-cycle, whereas any of type $ab \sim cd$ would impose a 4-cycle through the path $a, ab, cd$, but $N(a) \cap N(cd) = \emptyset$).

To prepare for working with $L_3$, consider a complete graph $K_4$ with $V(K_4) = \{1,2,3,4\}$.  The edges of this $K_4$ have the same names as the labels of vertices in $L_2$.  So, we may associate each vertex $x \in L_3$ with the subgraph $C_x$ of $K_4$ which has edge set $N(x) \cap L_2$ (i.e. all neighbours of $x$ which are in $L_2$).  Let $\mathcal{C}$ denote the set of subgraphs of $K_4$ associated with vertices in $L_3$.  

Suppose first that $C_x \in \mathcal{C}$ contains all of the edges $12$, $13$, and $14$ from our $K_4$.  Then the corresponding vertex $x$ in $L_3$ together with the vertices $1, 12, 13, 14$ form a subgraph isomorphic to $K_{3,2}$, giving us a contradiction.  It follows from this reasoning that every graph in $\mathcal{C}$ has maximum degree at most 2.  

By applying the completion property to the two paths in $X$ given by $12, 1, 13$ and $12, 1, 14$ we deduce that $L_3$ must contain a vertex $x$ adjacent to $12$ and $13$ and another vertex $x'$ adjacent to $12$ and $14$.  Since $X$ is 4-regular, these are the only vertices in $L_3$ adjacent to $12$.  It follows from this property that every graph in $\mathcal{C}$ is a cycle, and further the cycles in $\mathcal{C}$ cover every edge of the $K_4$ exactly twice.

The only possibilities are for $\mathcal{C}$ to consist of the three 4-cycles in $K_4$, or the four 3-cycles in $K_4$.  In the former case, every vertex in $L_3$ has all of its neighbours in $L_2$ and we find that $X$ is isomorphic to $\coHea$.  In the latter case, each vertex $x$ in $L_3$ is associated with a 3-cycle $C_x$ in $K_4$ and we may label each such $x$ with $V(C_x)$ (a 3-element subset of 1234).  Using these labels, we see that the graph induced by $L_0 \cup L_1 \cup L_2 \cup L_3$ is isomorphic to the graph obtained from $Q_4$ by removing a vertex.  Now applying the completion property to the edges between $L_3$ and $L_4$ implies that $X \cong Q_4$.  
\end{proof}

\begin{lemma}\label{lem:K42}
 Let $X = (V,E)$ be a graph in $\cF$ with a subgraph isomorphic to $K_{4,2}$.  Then $X$ is either isomorphic to $K_{4,4}$,  or $C_m^{ (2) }$ for some $m \ge 5$ (possibly $m= \infty$), or $2 \times H$ where $H$ is an arc-transitive 4-regular graph.
 \end{lemma}
 
\begin{proof} 
Call two vertices $u,v \in V$ \emph{clones} if $N(u) = N(v)$ and note that this implies $\{u,v\} \not\in E$.  Clones give an equivalence relation on $V$ and we let $V^{\bullet}$ denote the corresponding partition of $V$ into equivalence classes (so every $S \in V^{\bullet}$ is a maximal set of clones).  Let $X^{\bullet}$ be the graph with vertex set $V^{\bullet}$ and two sets $S,T \in V^{\bullet}$ adjacent in $X^{\bullet}$ if there is a vertex in $S$ adjacent to a vertex in $T$ (in the original graph $X$).  Note that $X^{\bullet}$ will still be connected and edge-transitive since these properties are inherited from $X$.  

First suppose that in the original graph $X$, every vertex has another clone.  It then follows that $X^{\bullet}$ has maximum degree at most 2, so $X^{\bullet}$ must be isomorphic to either $K_2$, $C_m$ for some $m \ge 3$, or a two-way infinite path.  In the first case $X \cong K_{4,4}$ and in the third we have $X \cong C_{\infty}^{(2)}$.  In the second case we find that $X \cong C_m^{(2)}$ for some $m \ge 3$.  If $m=3$ we have a contradiction to the assumption that $X$ has girth 4, and for $m=4$ note that $C_4^{(2)} \cong K_{4,4}$.  So in this case we must have $m \ge 5$ as desired.

Since $X$ contains a subgraph isomorphic to $K_{4,2}$, we know that it has distinct vertices which are clones.  As not every vertex has another clone, $X$ must not be vertex transitive.  It then follows from the edge-transitivity of $X$ that the vertices of $X$ fall into two orbits $V_0$ and $V_1$ under the action of $\Aut(X)$ and further $(V_0, V_1)$ is a bipartition of $X$.  Assume that the vertices in $V_1$ have no distinct clone and let $(V_0^{\bullet}, V_1^{\bullet})$ be the bipartition of $X^{\bullet}$ corresponding to $(V_0,V_1)$.  Since $X$ has degree 4 and it is not isomorphic to $K_{4,4}$, for every $v \in V_1$, the neighbours of $v$ must fall into two sets of clones, each of size two.  So, the graph $X^{\bullet}$ is bipartite with every vertex in $V_0^{\bullet}$ of degree 4 and every vertex in $V_1^{\bullet}$ of degree 2.  Now define $Y$ to be the graph with vertex set $V_1^{\bullet}$ and two vertices $u,v \in V_1^{\bullet}$ adjacent if $u,v$ have a common neighbour in $X^{\bullet}$.  Our construction implies that $X \cong 2 \times Y$; furthermore, the edge-transitivity of $X^{\bullet}$ implies that $Y$ is arc-transitive.  
\end{proof}

We are now ready to prove our main result.

\begin{proof}[Proof of Theorem~\ref{thm:main}]
Let $X$ be a graph in $\cF_k$.  If $X$ has a subgraph isomorphic to $K_{4,2}$ then by Lemma \ref{lem:K42}, it is isomorphic to $K_{4,4}$, to $C_m^{ (2) }$ for some $m \ge 5$ (possibly $m= \infty$) or to $2 \times H$ where $H$ is an arc-transitive 4-regular graph.  If it has a $K_{3,2}$ subgraph but no $K_{4,2}$ subgraph, then by Lemma \ref{lem:K32} we find that $X$ is isomorphic to $K_{5,5}-M$.  If $X$ has no $K_{3,2}$ subgraph and $k \ge 3$ then by Lemma \ref{lem:Q4coHea} it is isomorphic to either $Q_4$ or $\coHea$.  The only remaining possibilities are $k=1$ and $k=2$ which were dealt with in Sections \ref{squaresec} and \ref{quadsec}.  
\end{proof}


\end{document}